\documentclass[11pt,letterpaper]{amsart}
\usepackage[small,nohug,heads=vee]{diagram}
\diagramstyle[labelstyle=\scriptstyle]
\usepackage[colorlinks, citecolor=blue, dvipdfm,pagebackref]{hyperref}
\diagramstyle[labelstyle=\scriptstyle]
\usepackage{graphicx}
\usepackage[all]{xy}
\usepackage{extarrows}
\usepackage{tikz}
\usetikzlibrary{shapes,arrows}
\usepackage{amsmath, tikz-cd}
\usetikzlibrary{shapes,arrows}

\newtheorem{theorem}{Theorem}[section]
\newtheorem{lemma}[theorem]{Lemma}

\theoremstyle{definition}
\newtheorem{definition}[theorem]{Definition}
\newtheorem{question}[theorem]{Question}

\newtheorem{proposition}[theorem]{Proposition}
\newtheorem{corollary}[theorem]{Corollary}
\newtheorem{remark}[theorem]{Remark}
\newtheorem{conjecture}[theorem]{Conjecture}

\theoremstyle{remark}

\newcommand{\be}{\begin{equation}}
\newcommand{\ee}{\end{equation}}

\numberwithin{equation}{section}

\begin{document}
\title[Fujita's conjectures and compactification of homology cells]{Compactification of homology cells, Fujita's conjectures and the complex projective space}
\author{Ping Li}
\address{School of Mathematical Sciences, Fudan University, Shanghai 200433, China}
\email{pinglimath@fudan.edu.cn\\
pinglimath@gmail.com}

\author{Thomas Peternell}
\address{Mathematisches Institut, Universit\"{a}t Bayreuth, 95440 Bayreuth, Germany}
\email{thomas.peternell@uni-bayreuth.de}
%    \thanks will become a 1st page footnote.
%\author{}
%\address{}
%\email{}
\thanks{The first-named author was partially supported by the National
Natural Science Foundation of China (Grant No. 12371066).
 }

%    General info


\begin{abstract}  We show that a compact K\"ahler manifold $M$ containing a smooth connected divisor $D$ such that $M \setminus D$ is a homology cell, e.g., contractible, must be projective space
with $D$ a hyperplane, provided $\dim M \not \equiv 3 \pmod 4$. This answers conjectures of Fujita in these dimensions.
\end{abstract}

\maketitle

\section{Introduction}
Recall the following two basic problems in Hirzebruch's famous 1954 problem list; Problems $27$ and $28$ in \cite[p.231]{Hi54}.

A pair $(M,D)$ is called an (analytic) \emph{compactification} of an open $n$-dimensional complex manifold $U$ if $M$ is an $n$-dimensional compact complex manifold and $D\subset M$ an analytic subvariety such that $M\backslash D$ is biholomorphic to $U$. Such a compactification $(M,D)$ is called \emph{smooth} resp. \emph{K\"{a}hler} if $D$ is smooth or $M$ is K\"{a}hler, respectively.

In \cite[Problem 27]{Hi54}, Hirzebruch raised the problem of classifying the
compactifications $(M,D)$ of $U=\mathbb{C}^n$ with second Betti number $b_2(M)=1$. The Betti number condition is equivalent to the irreducibility of the subvariety $D$.

The standard example is $(M,D)=(\mathbb{P}^n,\mathbb{P}^{n-1})$, where $\mathbb{P}^{n-1}$ is some \emph{linearly} embedded subspace in $\mathbb{P}^n$. When $n=1$ or $2$, $(\mathbb{P}^n,\mathbb{P}^{n-1})$ is the unique example (\cite{RvdV60}). The only smooth compactification for $n=3$ is $(\mathbb{P}^3,\mathbb{P}^{2})$ (\cite[Thm 2.4]{BM}). When $n=3$ and $D$ is allowed to be singular, the classification is complicated (see \cite[p.781-782]{Hi87}, \cite{PS} and the references therein). When $n\leq 6$, the only K\"{a}hler smooth compactification is $(\mathbb{P}^n,\mathbb{P}^{n-1})$ (\cite{vdV}, \cite{Fu80-2}). We refer to \cite{BM} and \cite{PS} for a survey on these historical materials. It has been a long-standing open question if the only K\"{a}hler smooth compactification of $\mathbb{C}^n$ is $(\mathbb{P}^n,\mathbb{P}^{n-1})$ (\cite[Conjecture 3.2]{BM}), which was
recently answered in the affirmative by Chi Li and Zhou (\cite{LZ}), and by \cite{Pe} in the even dimensional case (in the more general setting where $U$ is a homology cell, discussed below).

In \cite[Problem 28]{Hi54}, Hirzebruch further raised the problem of determining all complex or K\"{a}hler structures on (the underlining differentiable manifold of) $\mathbb{P}^n$. For the complex structure the uniqueness is well-known for $n=1$ and now known for $n=2$ (\cite{yau}, \cite[\S3]{De}).
In dimension $n \geq 3$, the problem is wide open although there are some partial results in dimension $3$.
 Hirzebruch observed that (\cite[p.223]{Hi54}) the uniqueness of complex structure on $\mathbb{P}^3$ would imply the nonexistence of complex structures on $S^6$ (see \cite[Prop.3.1]{To} for a detailed proof). In fact, blowing up a point in a potential complex structure on $S^6$ one obtains a complex structure $M$ on $\mathbb P_3$ such that $M \setminus D$ is contractible, where $D$ is the exceptional divisor. Here $M \setminus D$ is of course far from being Stein.

For the K\"{a}hler structure the problem has been solved due to the following uniqueness result: a K\"{a}hler manifold homeomorphic to $\mathbb{P}^n$ must be biholomorphic to $\mathbb{P}^n$. Its proof combines a result of Hirzebruch and Kodaira (\cite{HK}), the homeomorphic invariance of rational Pontrjagin classes due to Novikov (\cite{No}), and the Miyaoka-Yau Chern number inequality (\cite{CO},\cite{yau}). A detailed proof can be found in the nice expository paper \cite{To} by Tosatti. When the dimension $n$ are small enough ($ n\leq 6$), the condition ``homeomorphic to $\mathbb{P}^n$" can be further relaxed to various weaker conditions (\cite{Fu80-2},\cite{LS},\cite{Wil},\cite{LW},\cite{ye}, \cite{Li16}, \cite{De}).
We remark that all these results rely on a well-known criterion due to Kobayashi and Ochiai (\cite{KO}): the Fano index of an $n$-dimensional Fano manifold is no more than $n+1$, with equality if and only if it is biholomorphic to $\mathbb{P}^n$.

Motivated by these two problems and some related results, Fujita raised in \cite[\S1]{Fu80-2} the following three closely related conjectures of increasing strength.
\begin{conjecture}[$A_n$]\label{An}
Let $U$ be an $n$-dimensional contractible complex manifold and $(M,D)$ a K\"{a}hler smooth compactification. Then $U\cong\mathbb{C}^n$ and $(M,D)$ is the standard example $(\mathbb{P}^n,\mathbb{P}^{n-1})$.
\end{conjecture}

\begin{conjecture}[$B_n$]\label{Bn}
Let $M$ be an $n$-dimensional projective manifold and $D$ a smooth ample divisor on $M$. Suppose that the natural homomorphism $H_k(D;\mathbb{Z})\longrightarrow H_k(M;\mathbb{Z})$ is bijective for $0\leq k\leq 2(n-1)$. Then $M\cong\mathbb{P}^n$ and $D$ is a hyperplane on it.
\end{conjecture}

\begin{definition} We will say that $M \setminus D$ is a \emph{homology cell} if  $H_k(D;\mathbb{Z})\longrightarrow H_k(M;\mathbb{Z})$ is bijective for $0\leq k\leq 2(n-1)$.
\end{definition}

\begin{conjecture}[$C_n$]\label{Cn}
Let $M$ be an $n$-dimensional Fano manifold such that its integral cohomology ring $H^{\ast}(M;\mathbb{Z})\cong H^{\ast}(\mathbb{P}^n;\mathbb{Z})$. Then $M\cong\mathbb{P}^n$.
\end{conjecture}
Conjecture $(A_n)$ is much stronger than the aforementioned folklore conjecture solved in \cite{LZ}. Fujita showed that $(C_n)$ implies $(B_n)$ and $(B_{n+1})$, $(B_n)$ implies $(A_n)$ (\cite[p.233]{Fu80-2}), and $(C_n)$ is true for $n\leq 5$ (\cite[Thm 1]{Fu80-2}).

\emph{The major purpose} in this note is to show the following result, which solves Conjecture $(B_n)$ in the affirmative for $n\not\equiv3 \pmod{4}$. Namely, we have
\begin{theorem}\label{firstmainresult}
Let $M$ be an $n$-dimensional projective manifold, $D$ a smooth divisor on $M$, such that $M \setminus D$ is a homology cell.
\begin{enumerate}
\item
If $n\not\equiv3 \pmod{4}$, $M\cong\mathbb{P}^n$ and $D$ is a hyperplane.

\item
If $n\equiv3 \pmod{4}$, either $M\cong\mathbb{P}^n$ and $D$ is a hyperplane, or $M$ is a Fano manifold with Picard number one, of index $\frac{1}{2}(n+1)$, and $D\in|\mathcal{O}_M(1)|$. Moreover, in the latter case $c_{n-1}(M)=n(n+1)x^{n-1}$ and $c_{n-2}(D)=n^2x_0^{n-2}$, where $x$ and $x_0$ are positive generators of $H^2(M;\mathbb{Z})$ and $H^2(D;\mathbb{Z})$ respectively.
\end{enumerate}
\end{theorem}
\begin{remark}
The condition ampleness in Conjecture \ref{Bn} turns out to be redundant. When $n\equiv3 \pmod{4}$, the Fano index of $D$ in the non-standard case is $\frac{1}{2}(n+1)-1=\frac{1}{2}(n-1)$, which is half of $\dim D$. There are structural results for Fano $n$-folds of second Betti number $b_2\geq2$ with Fano indices $\frac{1}{2}(n+1)$ ($n$ odd) and $\frac{1}{2}n$ ($n$ even) respectively (\cite{Wis91}, \cite{Wis93}). We suspect that the non-standard case may not occur. The last section discusses this case in detail.
\end{remark}

As a consequence, Conjecture $(A_n)$ is also true when $n\not\equiv3 \pmod{4}$. Indeed, Conjecture $(A_n)$ can be solved in a more general setting which is a reformulation of the previous theorem
via some basic topological considerations.
\begin{theorem}\label{secondmainresult}
Let $U$ be an $n$-dimensional open complex manifold which is homology trivial (namely the reduced homology $\widetilde{H}_i(U;\mathbb{Z})=0$ for all $i$), and $(M,D)$ its K\"{a}hler smooth compactification. Then the conclusions in Theorem \ref{firstmainresult} hold.
In particular, Conjecture \ref{An} is true whenever $n\not\equiv3 \pmod{4}$.
\end{theorem}
\begin{corollary}\label{main coro}
Let $M$ be a compact K\"{a}hler manifold of dimension $n\not\equiv3\pmod{4}$, and $D\subset M$ a smooth (connected) divisor such that $M\backslash D\cong\mathbb{C}^n$. Then $M\cong\mathbb{P}^n$ and $D$ is a hyperplane on it.
\end{corollary}
\begin{remark}
As already mentioned,
Corollary \ref{main coro} was shown without assumption on the dimension by Chi Li and Zhou in \cite{LZ}, whose methods are completely different.
\end{remark}

We will further give an application of Theorem \ref{firstmainresult}.
Sommese showed in \cite{So} that there are severe restrictions on a projective manifold if it can be realized as an ample divisor in some other projective manifold. Fujita further improved in \cite{Fu80-1} some of Sommese's results and answered some questions and conjectures raised in \cite{So}. As remarked in \cite[Remark 2]{Fu80-2}, a positive answer to Conjecture \ref{Bn} would lead to the following result, which solves \cite[Question 4.5]{Fu80-1} and gives a sharpened form of \cite[p.64, Prop.5]{So}.
\begin{theorem}\label{application-fibration}
Let $D$ be a smooth ample divisor in a projective manifold $M$ and $f: D\longrightarrow S$ a holomorphic mapping of maximal rank at every point onto a compact complex manifold $S$. Assume $f$ extends holomorphically to $F: M\longrightarrow S$. Then $\dim M\geq2\dim S$. If moreover $\dim M=2\dim S$ or $2\dim S+1$, and $n\not\equiv3 \pmod{4}$, both $f$ and $F$ are fiber bundles with fibers being isomorphic to projective spaces so that each fiber of $f$ is a hyperplane on the respective fiber of $F$.
\end{theorem}
\begin{remark}
When $\dim D-\dim S\geq2$, it turns out that the extension $F$ always exists (\cite[p.61, Prop.3]{So}).
\end{remark}

This note is structured as follows. Since \cite[Thm 2]{Fu80-2} is crucial to establishing the main results, a detailed proof will be provided in Section \ref{preliminary} for the reader's convenience. In Section \ref{proof of Peternell results} we apply Fujita's result and a Chern number identity to narrow down the first Chern classes of the pair $(M,D)$ in question to two possible cases. Then Section \ref{proof of main results} is devoted to the proof of our main results. Finally in Section \ref{approcach to n=3 mod4} we discuss approaches to the open case $\dim M \equiv 3 \pmod 4$ and give some partial results.

The current note is based on a combination of the two recent arXiv papers \cite{Pe} and \cite{Li25} due to the second-named and the first-named author respectively. Section \ref{proof of Peternell results}, especially Proposition \ref{Peternell proposition}, is taken from \cite{Pe}, where the second-named author solves Conjecture \ref{Bn} when $n$ are even. Section \ref{proof of main results} is taken from \cite{Li25}, where the first-named author pushes forward the result to include the case $n\equiv1\pmod{4}$.

\section{Preliminaries}\label{preliminary}
In this section we assume first that $M$ is a $2n$-dimensional closed (connected) oriented (topological) manifold with $n\geq2$, and $D\overset{i}{\hookrightarrow}M$ a $2(n-1)$-dimensional closed (connected) oriented submanifold in $M$. Let
$$P_M(\cdot):=(\cdot)\cap[M]:~H^k(M;\mathbb{Z})\overset{\cong}{\longrightarrow}
H_{2n-k}(M;\mathbb{Z})\quad(0\leq k\leq 2n)$$
be the Poincar\'{e} duality of $M$, where $[M]$ is the fundamental class of $M$ determined by the orientation and $``\cap"$ the cap product. Let \be\label{poincare dual to x}x:=P_M^{-1}\big(i_{\ast}([D])\big)\in H^2(M;\mathbb{Z})\ee be the Poincar\'{e} dual of $D$ in $M$, i.e., $x\cap[M]=i_{\ast}([D])$. Let $x_0:=i^{\ast}(x)\in H^2(D;\mathbb{Z})$ be the restriction of $x$ to $D$.

With the notation above understood, we have the following lemma.
\begin{lemma}\label{lemma-isomorphism}
Assume that the natural homomorphism
$$H_k(D;\mathbb{Z})\overset{i_{\ast}}{\longrightarrow} H_k(M;\mathbb{Z})$$ induced by $D\overset{i}{\hookrightarrow}M$ is bijective for $0\leq k\leq 2(n-1)$, i.e.,
$M \setminus D$ is a homology cell.
\begin{enumerate}
\item
The even-dimensional cohomology rings of $M$ and $D$ are given by
$H^{2\ast}(M;\mathbb{Z})=\mathbb{Z}[x]/(x^{n+1})$ and $H^{2\ast}(D;\mathbb{Z})=\mathbb{Z}[x_0]/(x_0^{n})$.

\item
If furthermore the first Betti number $b_1(M)=0$, their cohomology rings are given by
$H^{\ast}(M;\mathbb{Z})=\mathbb{Z}[x]/(x^{n+1})$ and $H^{\ast}(D;\mathbb{Z})=\mathbb{Z}[x_0]/(x_0^{n})$.
\end{enumerate}
\end{lemma}
\begin{proof}
It is well-known that the isomorphisms $i_{\ast}$ imply isomorphisms on cohomology (\cite[Cor.3.4]{Ha}) \be\label{isomorphism on cohomology}H^k(M;\mathbb{Z})\overset{i^{\ast}}{\underset{\cong}{\longrightarrow}} H^k(D;\mathbb{Z}),\quad0\leq k\leq 2(n-1).\ee

Let
$\tilde{i}: H^k(D;\mathbb{Z})\overset{\cong}{\longrightarrow} H^{k+2}(M;\mathbb{Z})$ be the isomorphism so that the following diagram commutes:
\be\label{diagram1}\begin{diagram}
H^k(M;\mathbb{Z})\overset{i^{\ast}}{\underset{\cong}{\longrightarrow}} &H^k(D;\mathbb{Z}) &\overset{\tilde{i}}{\underset{\cong}{\longrightarrow}}  &H^{k+2}(M;\mathbb{Z})\\
&\dTo_{P_D}^{\cong} & &\dTo_{P_M}^{\cong}\\
&H_{2n-2-k}(D;\mathbb{Z}) &\overset{i_{\ast}}{\underset{\cong}{\longrightarrow}}
&H_{2n-2-k}(M;\mathbb{Z}),
\end{diagram}
\qquad \big(0\leq k\leq 2(n-1)\big).\ee
In fact, we set  $\tilde{i}:=P_M^{-1}\circ i_{\ast}\circ P_D$. We assert that the composition $\tilde{i}\circ i^{\ast}$ in (\ref{diagram1}) is of the following form
\be\label{diagram2}\begin{diagram}
\tilde{i}\circ i^{\ast}:\quad&H^k(M;\mathbb{Z})&\overset{\cong}{\longrightarrow}
&H^{k+2}(M;\mathbb{Z})\\
&\theta&\longmapsto
&\theta\cup x
\end{diagram},\quad 0\leq k\leq 2(n-1).\ee
Indeed we have
\be\begin{split}
\tilde{i}\circ i^{\ast}(\theta)=P_M^{-1}\circ i_{\ast}\circ P_D\big(i^{\ast}(\theta)\big)&=P_M^{-1}\circ i_{\ast}
\big(i^{\ast}(\theta)\cap[D]\big)\\
&=P_M^{-1}\big(\theta\cap i_{\ast}([D])\big)\\
&=\theta\cup x,
\end{split}\nonumber
\ee
where the last equality holds since
$$P_M(\theta\cup x)=(\theta\cup x)\cap[M]=\theta\cap
\big(x\cap[M]\big)\overset{(\ref{poincare dual to x})}{=}\theta\cap i_{\ast}([D]).$$
This completes the proof of (\ref{diagram2}). The isomorphisms $\tilde{i}\circ i^{\ast}$ in (\ref{diagram2}) and $i^{\ast}$ in (\ref{isomorphism on cohomology}) imply that the even-dimensional cohomology rings of $M$ and $D$ are as in Part $(1)$ in Lemma \ref{lemma-isomorphism}.

If moreover $b_1(M)=0$, the universal coefficient theorem yields that $H^1(M;\mathbb{Z})=0$, which leads to $H^{\text{odd}}(M;\mathbb{Z})=0$ and $H^{\text{odd}}(D;\mathbb{Z})=0$, still due to the isomorphisms in (\ref{diagram2}) and (\ref{isomorphism on cohomology}). This completes the proof of Part $(2)$ in Lemma \ref{lemma-isomorphism}.
\end{proof}

With Lemma \ref{lemma-isomorphism} in hand, we can now prove the following crucial result due to Fujita (\cite[Thm 2]{Fu80-2}), whose proof borrowed some ideas from that in \cite[Prop.5]{So}.
\begin{theorem}[Fujita]\label{Fujita theorem}
Let $M$ be an $n$-dimensional compact K\"{a}hler manifold with $n\geq2$, $D\overset{i}{\hookrightarrow}M$ a smooth divisor such that $M \setminus D$ is a homology cell.
Then
\begin{enumerate}
\item
$x:=c_1(\mathcal O_M(D))>0$ and so $\text{Pic}(M)=\mathbb{Z}L_D$, where $\mathcal O_M(D)$ is the line bundle determined by the divisor $D$.

\item
$H^{\ast}(M;\mathbb{Z})=\mathbb{Z}[x]/(x^{n+1})$ and $H^{\ast}(D;\mathbb{Z})=\mathbb{Z}[x_0]/(x_0^{n})$, where as before $x_0:=i^{\ast}(x)$.
\end{enumerate}
\end{theorem}
\begin{proof}
By Lemma \ref{lemma-isomorphism}, $\mathbb{Z}x=H^2(M;\mathbb{Z})\cong\mathbb{Z}$. Since $M$ is K\"{a}hler, either $x>0$ or $x<0$ (and hence $M$ is projective due to the Kodaira embedding theorem). Since $\mathcal O_M(D)$ has a global section, the case $x< 0$ cannot occur.
This completes the proof of Part $(1)$.

For Part (2), it suffices to show $b_1(M) = 0$. Following \cite[p.232]{Fu80-2}  and arguing by contradiction, the image of the Albanese map $\alpha$ must be a curve; otherwise $M$ would carry a holomorphic $2-$form contradicting $b_2(M) = 1$ by Hodge decomposition. But $\alpha(M) $ cannot be a curve. In fact, the preimage of a generic point in the curve $\alpha(M)$ would be an effective divisor that cannot be ample, contradicting $H^2(M;\mathbb{Z})\cong\mathbb{Z}$.
\end{proof}

\section{A key proposition}\label{proof of Peternell results}
The $\chi_y$-genus $\chi_y(M)\in\mathbb{Z}[y]$ of a compact complex manifold $M$ was introduced by Hirzebruch (\cite{Hi66}) (see Section \ref{approcach to n=3 mod4} for more details). When expanding $\chi_y(M)$ at $y=-1$, it is well-known that the constant term $\chi_{-1}(M)$ is the Euler number $c_n[M]$, where $n=\dim M$. A direct calculation using the Hirzebruch-Riemann-Roch theorem shows that the coefficient in front of the quadratic term $(y+1)^2$ is exactly
\be\label{HRR}\frac{n(3n-5)}{24}c_n[M]+\frac{1}{12}c_1c_{n-1}[M],\ee
and hence the Chern number $c_1c_{n-1}[M]$ can be determined by the Hodge numbers of $M$ via an explicit formula (\cite[p.18]{NR}, \cite[Thm 3]{LW}, \cite[Cor.3.4]{Sa}). As a consequence, we have (\cite[Cor.2.5]{LW})
\begin{lemma}\label{LW}
If $M$ is an $n$-dimensional compact K\"{a}hler manifold with the same Betti numbers as $\mathbb{P}^n$, then
$$c_1c_{n-1}[M]=c_1c_{n-1}[\mathbb{P}^n]=\frac{1}{2}n(n+1)^2.$$
\end{lemma}
\begin{remark}
The formula (\ref{HRR}), implicitly or explicitly, has been obtained by several independent articles with different backgrounds (\cite{NR}, \cite{LW}, \cite{Sa}). This kind of formula is a special case of a general phenomenon, which was called \emph{$-1$-phenomenon} and investigated by the first-named author in \cite{Li15} and \cite{Li17}. The reader may refer to \cite[\S 3.2]{Li19} for a summary on these materials.
\end{remark}

The next proposition is a key point in this note, which reduces the first Chern classes of the pair $(M,D)$ in question to two possible cases.
\begin{proposition}\label{Peternell proposition}
Let the conditions and notation be as in Theorem \ref{Fujita theorem}. Then
$$\text{$\big(c_1(M),c_1(D)\big)=\big((n+1)x,nx_0\big)$ or $\bigg(\frac{1}{2}(n+1)x,\frac{1}{2}(n-1)x_0\bigg)$}.$$
\end{proposition}
\begin{proof}
The three holomorphic vector bundles $i^{\ast}(TM)=TM\big|_D$ (the restriction to $D$ of the tangent bundle $TM$), $TD$ and the normal bundle $ND$ of $D$ in $M$ are related by the short exact sequence
\be\label{split}0\longrightarrow TD\longrightarrow i^{\ast}(TM)\longrightarrow ND\longrightarrow0.\ee
By Theorem \ref{Fujita theorem} the Chern classes $c_i(M)\in\mathbb{Z}x^i$ and $c_i(D)\in\mathbb{Z}x_0^i$. So $c_i(M)$ and $c_i(D)$ can be viewed as integers by abuse of notation. With this understood and taking the first and $(n-1)$-th Chern classes on $i^{\ast}(TM)$ via (\ref{split}), we have by adjunction that
\be\label{1}c_1(M)=c_1(D)+1,\quad c_{n-1}(M)=c_{n-1}(D)+c_{n-2}(D)=n+c_{n-2}(D)\ee
as $c_{n-1}(D)$ is the Euler number of $D$. Theorem \ref{Fujita theorem} implies that the Betti numbers of $M$ and $D$ are the same as those of $\mathbb{P}^n$ and $\mathbb{P}^{n-1}$ respectively. Hence Lemma \ref{LW} yields
\be\label{2}c_1(M)c_{n-1}(M)=\frac{1}{2}n(n+1)^2,\quad c_1(D)c_{n-2}(D)=\frac{1}{2}(n-1)n^2.\ee
By (\ref{2}) both $c_1(M)\neq0$ and $c_1(D)\neq0$. Hence
putting (\ref{1}) and (\ref{2}) together we have
\be\label{3}\frac{n(n+1)^2}{2c_1(M)}=n+\frac{(n-1)n^2}{2(c_1(M)-1)}.\ee
Solving (\ref{3}) leads to $c_1(M)=n+1$ or $\frac{1}{2}(n+1)$.
\end{proof}

\begin{remark}
The latter case in Proposition \ref{Peternell proposition} can occur only if $n$ is odd. So when $n$ is even, $\big(c_1(M),c_1(D)\big)=\big((n+1)x,nx_0\big)$ and hence Conjecture \ref{Bn} is true due to the well-known criterion of Kobayashi-Ochiai (\cite[p.32, Cor.]{KO}).
\end{remark}

\section{Proof of main results}\label{proof of main results}
Our starting point to improve on Proposition \ref{Peternell proposition} is the following fact, which should be known to some experts as it is an application of
some quite classical results in algebraic topology. In lack of  a reference  we provide a
detailed proof.
\begin{proposition}\label{homotopy equivalent}
Let $M$ be a simply-connected smooth closed $2n$-dimensional manifold such that its cohomology ring $H^{\ast}(M;\mathbb{Z})\cong H^{\ast}(\mathbb{P}^n;\mathbb{Z})$.
Then $M$ is homotopy equivalent to $\mathbb{P}^n$.
\end{proposition}
\begin{proof}
By a basic fact about the Eilenberg-MacLane space
$K(\mathbb{Z},2)=\mathbb{P}^{\infty}$ (\cite[Thm 4.57]{Ha}),
we have a bijection
$$H^2(M;\mathbb{Z})\longleftrightarrow[M,\mathbb{P}^{\infty}].$$  Then there exists a continuous map $$f:M\longrightarrow\mathbb{P}^{\infty},$$
 such that $f^{\ast}(u)=x$, where $H^{\ast}(\mathbb{P}^{\infty};\mathbb{Z})=\mathbb{Z}[u]$ and $H^{\ast}(M;\mathbb{Z})=\mathbb{Z}[x]/(x^{n+1})$.
 By the cellular approximation theorem (\cite[Thm 4.8]{Ha}), there exists another continuous map $g:~M\longrightarrow\mathbb{P}^{\infty}$ which is
 homotopic to $f$ such that $g(M)$ is contained in the $2n$-skeleton of $\mathbb{P}^{\infty}$ which is $ \mathbb{P}^{n}$. Since $g^{\ast}(u)=f^{\ast}(u)=x$, the map $g:~M\longrightarrow\mathbb{P}^n$
 induces an isomorphism on their cohomology rings:
\be\label{4}g^{\ast}:\quad H^{\ast}(\mathbb{P}^n;\mathbb{Z})=\mathbb{Z}[u]/(u^{n+1})
\overset{\cong}{\longrightarrow}H^{\ast}(M;\mathbb{Z})=
\mathbb{Z}[x]/(x^{n+1}).\ee
By (\ref{4}), $g^{\ast}(u^n)=x^n$ and therefore the degree of $g$ is $\pm1$. Choose suitable orientations $[M]$ and $[\mathbb{P}^n]$ on $M$ and $\mathbb{P}^n$. We have for each $0\leq k\leq n$,
$$g_{\ast}(x^{n-k}\cap[M])=g_{\ast}\big(g^{\ast}(u^{n-k})\cap[M]\big)
=u^{n-k}\cap g_{\ast}([M])=\pm u^{n-k}\cap[\mathbb{P}^n].$$
By Poincar\'{e} duality, this implies that $g$ induces isomorphisms on all integral homology groups:
\be\label{5}g_{\ast}:\quad H_{\ast}(\mathbb{P}^n;\mathbb{Z})
\overset{\cong}{\longrightarrow}H_{\ast}(M;\mathbb{Z}).\ee
Due to the simple-connectedness of $M$ and $\mathbb{P}^n$, and the fact (\ref{5}), the Whitehead theorem, \cite[Cor. 4.33]{Ha},
tells us that $g$ is a homotopy equivalence.
\end{proof}

Since a Fano manifold is simply-connected,
%(\cite[p.225]{Zh}),
Proposition \ref{homotopy equivalent} has an immediate consequence.
\begin{corollary}\label{coro}
A Fano manifold whose integral cohomology ring is the same as that of $\mathbb{P}^n$ must be homotopy equivalent to $\mathbb{P}^n$.
\end{corollary}
\begin{remark}
This implies that in Conjecture \ref{Cn}, the manifold in question is indeed homotopy equivalent to $\mathbb{P}^n$.
Libgober and Wood showed that a compact K\"{a}hler manifold homotopy equivalent to $\mathbb{P}^6$ is biholomorphic to
$\mathbb{P}^6$ (\cite[Thm 1]{LW}). So Conjecture \ref{Cn} is true when $n\leq 6$ (see also \cite[\S7]{De}).
\end{remark}

Now are are ready to prove the main results of this note.
\subsection*{Proof of Theorem \ref{firstmainresult}}
\begin{proof}
By Proposition \ref{Peternell proposition}, the Fano index $r$ of the Fano manifold $M$ is either $r = n+1$ or $ r = \frac{1}{2}(n+1)$.
If $r = n+1$, then we are done by the Kobayashi-Ochiai theorem (\cite[p.32, Cor.]{KO}).
Assume  that $ r = \frac{1}{2}(n+1)$. By Theorem \ref{Fujita theorem} and Corollary \ref{coro}, $M$ has the same homotopy type as $\mathbb{P}^n$.
The classical Wu formula (\cite[p.130]{MS}) says that Stiefel-Whitney classes are homotopy type invariants. Thus
\be\label{9}\frac{1}{2}(n+1)\equiv n+1 \pmod{2}\ee
as the first Chern class modulo two is the second Stiefel-Whitney class. We obtain from (\ref{9}) that $n\equiv3\pmod{4}$.
\end{proof}
\begin{remark}
The fact that the first Chern class modulo two is a homotopy type invariant was used throughout the arguments in \cite{LW}.
\end{remark}

\subsection*{Proof of Theorem \ref{secondmainresult}}
\begin{proof}
The general form of the Poincar\'{e}-Alexander-Lefschetz duality theorem (\cite[p.351]{Br}) says that, for compact subsets $B\subset A$ in $M$, we have
\be\label{duality}H^{2n-k}(M-B,M-A;\mathbb{Z})\cong
H_{k}(A,B;\mathbb{Z}),\quad 0\leq k\leq 2n.\ee
Taking $(A,B)=(M,D)$ in (\ref{duality}) yields
$$H_k(M,D;\mathbb{Z})\cong H^{2n-k}(M-D;\mathbb{Z})=0,\quad0\leq k\leq 2n-1,$$
as by the assumption in Theorem \ref{secondmainresult} the open manifold $U=M-D$ is homology trivial. This, via the homology long exact sequence for the pair $(M,D)$
$$\cdots\longrightarrow H_{k+1}(M,D;\mathbb{Z})\longrightarrow H_{k}(D;\mathbb{Z})\overset{i_{\ast}}{\longrightarrow}
H_{k}(M;\mathbb{Z})\longrightarrow H_k(M,D;\mathbb{Z})\longrightarrow\cdots,$$
yields that the the natural homomorphism $H_k(D;\mathbb{Z})\overset{i_{\ast}}{\longrightarrow}H_k(M;\mathbb{Z})$ is bijective for
$0\leq k\leq 2(n-1)$. Then Theorem \ref{secondmainresult} follows from Theorem \ref{firstmainresult}.
\end{proof}

\subsection*{Proof of Theorem \ref{application-fibration}}
\begin{proof}
The conclusion $\dim M\geq2\dim S$ was proved in \cite[p.64, Prop.5]{So}. Let $D_x:=f^{-1}(x)$ and $M_x:=F^{-1}(x)$ be the
respective fibers at $x\in S$. When $\dim M=2\dim S$ or $2\dim S+1$,  the natural homomorphism
$$H_k(D_x;\mathbb{Z})\longrightarrow H_k(M_x;\mathbb{Z})$$
is bijective for $0\leq k\leq 2\dim D_x$ (see the last two lines in \cite[p.64]{So}). Applying Theorem \ref{firstmainresult} to the pair $(M_x,D_x)$ yields the desired proof.
\end{proof}

\section{Approaches to the case $\dim M \equiv 3 \pmod 4$ }\label{approcach to n=3 mod4}
In this section, we present some possible approaches to the missing case $\dim M \equiv 3 \pmod 4$. After discussing the case when  $\mathcal O_M(1)$ has at least two independent sections, we set up a system of equations which conjecturally
lead to solving the remaining open case.

First, we have the following observation.
\begin{proposition}\label{h^0>1} Let $(M,D) $ be as in Theorem \ref{firstmainresult}.
 Let $\mathcal O_M(1) $ be the ample generator of ${\rm Pic}(M) \cong \mathbb Z$. If $n\equiv3\pmod 4$ and if $h^0\big(M,\mathcal O_M(1)\big) \geq 2$, then $M \cong \mathbb P^n$ and $D$ is a hyperplane.
\end{proposition}

\begin{proof}
By assumption we may pick a generic $\widetilde{D} \in \vert  \mathcal O_M(1) \vert $ such that $\widetilde{D}$ is smooth and different from $D$. Since
 $D \cdot\widetilde{D}$ has class one in $H^4(M,\mathbb Q)$, it follows from \cite[Prop.7.2]{Fu84} that $D \cap \widetilde{D}$ is smooth. Further, by the Mayer-Vietoris, the inclusions
 $$H_k(D \cap \widetilde{D};\mathbb Z) \to H_k(D;\mathbb Z)$$
 are bijective for $0 \leq k \leq 2(n-2)$.
 Since $\dim D$ is even, the conclusion follows from Theorem \ref{firstmainresult}.
 \end{proof}

We extract the assumption in Proposition \ref{h^0>1} as follows. Recall that the restriction map
 $$ H^0\big(M,\mathcal O_M(1)\big) \to H^0\big(D,\mathcal O_D(1)\big) $$
 is surjective due to the Kodaira vanishing theorem. Thus we are reduced to the following

\begin{question} {\rm Let $D$ be a Fano manifold of dimension $n=4k+2$ and Fano index $r = \frac{n}{2}$. Assume that ${\rm Pic}(D) \cong \mathbb Z$, $\mathcal O_D(1)$ the ample generator of ${\rm Pic}(D)$, $c_1(\mathcal O_D(1))^n  = 1$. Is
 $$ \dim  H^0\big(D,\mathcal O_D(1)\big) \ne 0?$$ }
 \end{question}
Of course, we may assume much more in our setting: its cohomology ring is the same as $\mathbb P^n$.

We next point out a possible numerical approach to solving the problem. We first set up some notation expanding the discussion at the
 beginning of Section \ref{proof of Peternell results}.

\begin{definition}\label{definition}
The $\chi_y$-genus of an $n$-dimensional compact complex manifold $X$, $\chi_y(M)$, is defined in terms of its Hodge numbers $h^{p,q}(M)$ by
 $$ \chi_y(M) := \sum_{p=0}^n \chi_p(M) y^p, $$
 where
 $$ \chi_p(M) = \sum_{q=0}^n (-1)^q h^{p,q}(M).$$
 Set further
 $$ A_k(M) = \frac{1}{(2k)!} \chi_y^{(2k)}(M)\big|_{y=-1},\quad 0\leq k\leq \lfloor\frac{n+2}{2}\rfloor.$$
 Namely, $A_k(M)$ is the coefficient in front of the term $(y+1)^{2k}$ when expanding $\chi_y(M)$ at $y=-1$.
 \end{definition}

\begin{remark}
Via the Hirzebruch-Riemann-Roch theorem,  the numbers $A_k(M)$ are determined by linear combinations of Chern numbers of $M$, and the first few ones for general $n$ have explicit formulas. Further, there is a recursive algorithm to determine them (see \cite[\S3.2]{Li19} and the references therein for a detailed summary on these facts).
\end{remark}

Note that $A_0(M) = c_n[M] $ and that
 $A_1(M) $ is a linear combination of $c_n[M]$ and $c_1c_{n-1}[M]$ \big(recall (\ref{HRR})\big).
 Furthermore, in addition to the two Chern numbers $c_n$ and $c_1c_{n-1}$, the new term arising from $A_2(M)$ is (see \cite[\S3.2]{Li19})
 $$ \big(c_1^2+3c_2\big)c_{n-2} - \big(c_1^3-3c_1c_2+3c_3\big) c_{n-3}.$$

The main point is now

 \begin{proposition} Let $M$ and $N$ be $n$-dimensional compact complex manifolds. If $M$ and $N$ have the same
 Hodge numbers, then
 $$ A_k(M) = A_k(N),\quad 0\leq k\leq \lfloor\frac{n+2}{2}\rfloor,$$
 \end{proposition}

This is simply due to the fact, that the numbers $A_k$ only depend on the Hodge numbers.
We conclude

 \begin{proposition}  Under the assumptions of Theorem \ref{firstmainresult}, the following equations hold.
 \begin{equation} A_k(M) = A_k(\mathbb P^n),\quad 0\leq k\leq \lfloor\frac{n+2}{2}\rfloor\nonumber  \end{equation}
 and
  \begin{equation} A_k(D) = A_k(\mathbb P^{n-1}),\quad \quad 0\leq k\leq \lfloor\frac{n+1}{2}\rfloor. \nonumber \end{equation}
  \end{proposition}

 Thus we obtain a system of equations
 \begin{eqnarray}\label{EQ}
\left\{\begin{array}{ll}
A_k(M) = A_k(\mathbb P^n),& 0\leq k\leq \lfloor\frac{n+2}{2}\rfloor,\\
~\\
A_k(D) = A_k(\mathbb P^{n-1}),& 0\leq k\leq \lfloor\frac{n+1}{2}\rfloor,\\
~\\
c_i(M) = c_i(D) + c_{i-1}(D), & \ 1 \leq i \leq n-1.
\end{array}
\right.
\end{eqnarray}

The hope - in a strong version -  is now that this system of equations has only one integer solution, namely $$c_i(M) = c_i(\mathbb P^n) = \binom{n+1}{i}$$ for all $i$
  (and thus $c_i(D) = c_i(\mathbb P^{n-1})$).

  It would be however be sufficient to prove a weak version, namely that there is no integer solution with $c_1(M) = \frac{n+1}{2}$.

Recall that the the system of equations (\ref{EQ}) for $k = 0$ and $k = 1$ simply gives
$$ c_n(M) = c_n(\mathbb P^n) = n+1;$$
$$ c_{n-1}(D) = c_{n-1}(\mathbb P^{n-1}) = n; $$
$$ c_1(M) \cdot c_{n-1}(M) = c_1(\mathbb P^n) \cdot  c_{n-1}(\mathbb P^n) = \frac{1}{2} n (n+1)^2;$$
$$ c_1(D) \cdot c_{n-2}(D) = c_1(\mathbb P^{n-1}) \cdot  c_{n-2}(\mathbb P^{n-1}) = \frac{1}{2} (n-1) n^2.$$
For $n = 5$, we get two more equations
 \be\begin{split}
 &\big(c_1^2(M) + 3c_2(M)\big) \cdot c_{3}(M) - \big(c_1^3(M) - 3c_1(M) \cdot c_2(M) + 3c_3(M)\big) \cdot c_{2}(M)\\
   =&  \big(c_1^2(\mathbb P^5) + 3c_2(\mathbb P^5)\big) \cdot c_{3}(\mathbb P^5) - \big(c_1^3(\mathbb P^5) - 3c_1(\mathbb P^5) \cdot c_2(\mathbb P^5) + 3c_3(\mathbb P^5)\big) \cdot c_{2}(\mathbb P^5)\end{split}\nonumber\ee
   and
 \be\begin{split}
 &\big(c_1^2(D) + 3c_2(D)\big) \cdot c_{2}(D) - \big(c_1^3(D) - 3c_1(D) \cdot c_2(D) + 3c_3(D)\big) \cdot c_{1}(D)\\
  =&  \big(c_1^2(\mathbb P^4) + 3c_2(\mathbb P^4)\big) \cdot c_{2}(\mathbb P^4) - \big(c_1^3(\mathbb P^4) - 3c_1(\mathbb P^4) \cdot c_2(\mathbb P^4) + 3c_3(\mathbb P^4)\big) \cdot c_{2}(\mathbb P^4).
  \end{split}\nonumber\ee
  These equations in combination with the the third one in (\ref{EQ}) do not have an integer solution in case $c_1(M) = 3$.

Here are some further information on the Chern classes of $M$.

\begin{proposition} $\sum_{k=0}^{n} (-1)^k c_k(M) = (-1)^n.$
 \end{proposition}

 \begin{proof}
 Observe first that $\chi(M \setminus D) = 1$ by the Mayer-Vietoris sequence.
 By Iitaka's
 log version of Hopf's theorem (see \cite[p.7, Prop. 2]{Ii} or \cite[Thm 3]{Nor}),
 $$ \chi(M \setminus D)  = (-1)^n c_n(\Omega_M(\log D)), $$
 we conclude that
 $ c_n(\Omega^1_M(\log D)) = {(-1)}^n.$
 Since $c_k(\mathcal O_D) = 1$ (here we identify $\mathcal O_D $ and $i_*(\mathcal O_D)$, where
 $i: D \to M$ is the inclusion), the formula follows.

 \end{proof}

\section*{Acknowledgements}
The first author thanks Yang Su for communicating some arguments in Proposition \ref{homotopy equivalent} several years ago. The second author thanks Daniel Barlet, Fr\'ed\'eric Campana, Baohua Fu, Andreas H\"oring, Mihai P\v{a}un
and Frank Olaf Schreyer for very valuable discussions. Both authors thank the referee for the careful reading and many constructive suggestions, which helped to improve the presentation.

\textbf{Data availability statement} No new data were created or analysed in this study. Data sharing is not applicable
to this article.

\textbf{\large{Declarations}}

\textbf{Conflict of interest} On behalf of all authors, the corresponding author states that there is no Conflict of
interest.

\end{document}